\documentclass[a4paper,11pt]{article}
\usepackage[utf8]{inputenc}
\usepackage[T1]{fontenc}

\usepackage{wrapfig}
\usepackage{amsthm,amsmath}
\usepackage{tikz}
\usetikzlibrary{graphs,graphs.standard}
\tikzgraphsset{edges={draw,semithick}, nodes={circle,draw,semithick}}
\usepackage{mathrsfs,amssymb,amsfonts} 
\usepackage{enumitem}
\usepackage{fullpage}
\usepackage{caption}
\usepackage{hyperref, enumerate}
\usepackage[babel]{microtype}
\usepackage[english]{babel}
\usepackage[capitalise]{cleveref}

\usepackage{thmtools}
\usepackage{mathtools, comment}
\usepackage{amssymb}
\usepackage[nomath]{lmodern}
\usepackage{graphicx}
\usepackage{pgf,tikz,tkz-graph,subcaption}
\usetikzlibrary{arrows,shapes}
\usetikzlibrary{decorations.pathreplacing}
\usepackage{tkz-berge}
\usepackage{enumitem}
\usepackage[normalem]{ulem}
\usepackage{hyperref}
\hypersetup{colorlinks = true, linkcolor = blue, citecolor = blue, urlcolor = blue}

\newcommand*{\ceilfrac}[2]{\mathopen{}\left\lceil\frac{#1}{#2}\right\rceil\mathclose{}}
\newcommand*{\floorfrac}[2]{\mathopen{}\left\lfloor\frac{#1}{#2}\right\rfloor\mathclose{}}

\newcommand*{\abs}[1]{\lvert #1\rvert}

\newcommand{\diam}{diam}

\allowdisplaybreaks

\usepackage[margin=1in]{geometry}
\newtheorem{defi}{Definition}

\newtheorem{cor}[defi]{Corollary}
\newtheorem{thm}[defi]{Theorem}
\newtheorem{lem}[defi]{Lemma}

\newtheorem{prop}[defi]{Proposition}
\newtheorem{exam}[defi]{Example}

\newtheorem{remark}[defi]{Remark}
\newtheorem{claim}[defi]{Claim}
\newcommand*{\myproofname}{Proof}
\newenvironment{claimproof}[1][\myproofname]{\begin{proof}[#1]}{\end{proof}}

\title{Uniform \v Solt\'es' hypergraphs and \v Solt\'es' weighted graphs}

\author{Stijn Cambie
 \thanks{Department of Computer Science, KU Leuven Campus Kulak-Kortrijk, 8500 Kortrijk, Belgium. Supported by a postdoctoral fellowship by the Research Foundation Flanders (FWO) with grant number 1225224N. Email: \protect\href{mailto:stijn.cambie@hotmail.com}{\protect\nolinkurl{stijn.cambie@hotmail.com}}} \and Ajay Tiwari \thanks{Department of Mathematics, KU Leuven, Celestijnenlaan, 200B, Leuven.  Email: \protect\href{mailto:ajaytiwari111@hotmail.com}{\protect\nolinkurl{ajaytiwari111@hotmail.com}} }}

\begin{document}
\parindent=0cm
\maketitle

\begin{abstract}
    A \v Solt\'es' hypergraph is a hypergraph for which the removal of any of its vertices does not change its total distance.
    We prove that every uniform \v Solt\'es' hypergraph has order at least $10$, there exist uniform \v Solt\'es' hypergraphs for almost every order or uniformity, and there exist a non-regular uniform \v Solt\'es' hypergraph.
    By also providing infinitely many weighted \v Solt\'es' graphs, we conclude that  \v Solt\'es' problem can be answered positively for the most natural generalisations of graphs.
\end{abstract}

\section{Introduction}
The \v Solt\'es' problem~\cite{Soltes91}, which asks about graphs for which vertex removal does not change the total distance, has been studied extensively in the context of graph theory, with recent work extending this problem to hypergraphs~\cite{Cambie24Hyper}. In this paper, we build upon this prior research, focusing specifically on uniform hypergraphs. We assume that the reader is familiar with basic terminology; for those unfamiliar, we refer to \cite[Subsec.~1.1]{Cambie24Hyper}.
    
    We address $4$ subquestions posed in~\cite[Ques.~13]{Cambie24Hyper}.
    The main results are summarised as follows:
    \begin{itemize}
        \item the smallest uniform \v Solt\'es' hypergraphs have order $10$,
        \item for every $n \ge 10$, there exists a uniform \v Solt\'es' hypergraph of order $n$,
        \item for most integers $k \ge 4$, there exists a $k$-uniform \v Solt\'es' hypergraph,
        \item there exists a non-regular uniform \v Solt\'es' hypergraph.
    \end{itemize}


    

    The problem of \v Solt\'es has been considered for various extensions of graphs, and for each of these, one has concluded that there is an abundancy for the analogue of \v Solt\'es' problem.

    We first recall a few definitions from~\cite{Spiro22}.
    A \emph{signed graph} is a pair \((G, f)\) where \(G\) is a graph and \(f\) is a function from \(E(G)\) to \(\{ \pm 1 \}\), which is called a \emph{signing}.
If \( P \) is a path in \( G \) and \(f \) is a signing of \( G \), 
\(
f(P) := \sum_{e \in P} f(e).
\)
For \( u, v \in V(G) \), the \emph{signed distance} is defined as
$
d_{G,f}(u,v) = \min_P \abs{f(P)}$
where the minimum ranges over all \( uv \)-paths \( P \).
The Wiener index \( W_f(G) \) of the signed graph \( (G, f) \) equals
\[
W_f(G) =  \sum_{\{u,v\} \subset V(G)} d_f(u,v).
\]

Spiro~\cite[Cor.~1.2]{Spiro22} proved that the \v Solt\'es' problem has an abundance of solutions for signed graphs, including generalisations involving the removal of subsets of vertices. Specifically, the following result holds: 
\begin{thm}[\cite{Spiro22}]
    If \( n \) is sufficiently large and \( G \) is an \( n \)-vertex graph with minimum degree at  
least \( \frac{2n}{3} \), then there exists a signing \( f \) of \( G \) such that the signed graph $(G, f)$ has \v Solt\'es' property in a strong form, \[
W_f(G) = W_f(G - v) = 0 \quad \text{for all } v \in V(G).
\]
\end{thm}

Cambie~\cite[Thm.~1]{Cambie25digraph} proved abundancy for the \v Solt\'es' problem in the digraph (directed graph) setting, as stated below, and gave an example of a large \v Solt\'es' digraph with trivial automorphism (no two vertices belong to the same vertex orbit).
\begin{thm}[\cite{Cambie25digraph}]
    For every $z \in \mathbb Z,$ there are infinitely many digraphs $D$ for which $W(D)-W(D \setminus v)=z$ for every $v \in V.$
\end{thm}
While not proven in detail, the same should be true for oriented graphs.

We also note that for weighted graphs (with positive weights) \v Solt\'es' problem has infinitely many solutions. 
A concrete example is drawn in~\cref{fig:SoltesWG}.

\begin{exam}\label{exam:weightedSG}
    For every $k \ge 20$, let $n=2k$ and let $G$ be the prism graph $C_n \square K_2$, for which the edges of the two induced cycles $C_n$ have all weight $1$ and the remaining edges have weight $x=\frac{2k^2-6k+16}{k^2-9k+12}$.
    Then $W(G)=W(G \setminus v)$ for every $v \in V(G).$
\end{exam}

\begin{proof}
    Note that $2<x\le 3$ since $k \ge 20$ and the weighted graph is vertex-transitive.
    An elementary computation shows that $W(G)=4W(C_n)+n^2x=\frac{n^3}{2}+n^2x$ and thus $\sigma(v)=\frac{n^2}{2}+nx=2k^2+2kx$ for every $v \in v(G).$
    Furthermore, the detours are exactly between vertices $u,w \in V(G) \setminus v$ on the same induced $C_n$ as $v$, for which the shortest path in $G$ (and thus the $C_n$) uses $v$. The detour is $2$ or $4$ if $d(u,v)=k-1$ or $d(u,v)=k-2$ respectively, or $2x$ if $d(u,v) \le k-3.$ 
    Counting the pairs of vertices with these properties, implies that the sum of detours equals
    $$\sum_{\substack{u, w \in V\setminus \{v\}}} \big[d_{G \setminus v}(u, w) - d(u, w)\big]=(k-3)(2+4)+2+2x\binom{k-3}{2}=6k-16+x(k-3)(k-4).$$
    By the choice of $x,$ the sum of detours and the transmission of $v$ are equal and thus $W(G)=W(G \setminus v)$.
\end{proof}

\begin{remark}
    By multiplying the weights of~\cref{exam:weightedSG} with $\frac{k^2-9k+12}{\gcd(2k^2-6k+16,k^2-9k+12)}$, one can conclude there are infinitely many examples of weighted graphs with weights in $\mathbb Z^+$ which satisfy the \v Solt\'es' property and for which the greatest common divisor of weights is $1$ (to prevent $C_{11}$ with $11$ times the same arbitrary weight to be considered as an infinite family).
    The cycle $C_{10}$ with alternating edges having weight $0$ and $1$ can be considered as a small example.
\end{remark}

We conclude that in the most natural generalisations of graphs\footnote{This is a slightly subjective statement. Some people may consider e.g. matroids, in which case the problem is again quite different.} (digraphs, uniform hypergraphs, weighted graphs) \v Solt\'es' problem has infinitely many solutions.

\begin{figure}[h]
    \centering
    \begin{tikzpicture}[scale=0.668]

\foreach \x in {0,1,2,...,39}{
 \draw[thick, blue] (\x*9:5) -- (9*\x:6.5) ;
    \foreach \r in {5,6.5}{
\draw (\x*9:\r) -- (9*\x+9:\r) ;
\draw[fill] (\x*9:\r) circle (0.15);
}

}

\draw (4.5:7) node [black]{\huge $1$}; 
\draw (4.5:4.5) node [black]{\huge $1$}; 
\draw (4.5:5.75) node [black]{\huge \textcolor{blue}{$3$}}; 

\end{tikzpicture}
    \caption{A weigthed \v Solt\'es' graph of order $80$}
    \label{fig:SoltesWG}
\end{figure}
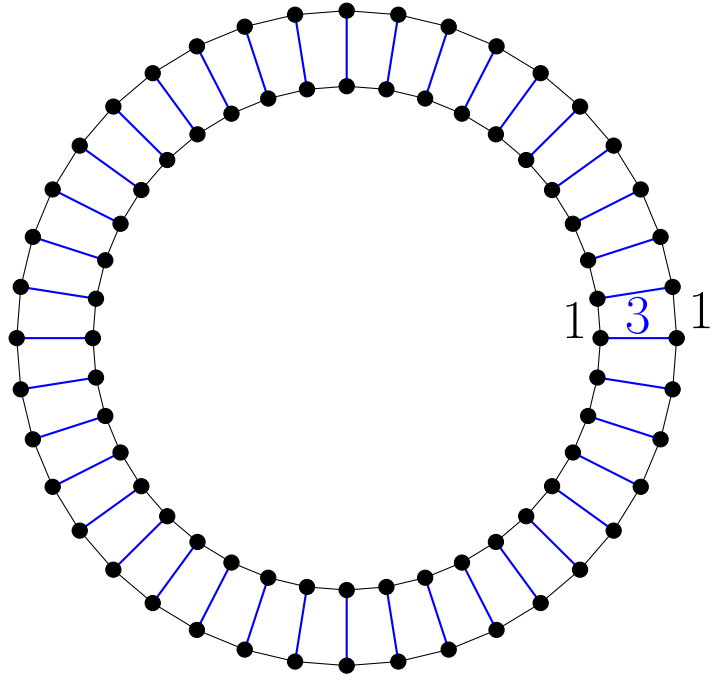

For resolving the graph case, we refer to the survey~[Sec.~6]\cite{KST23} and the paper~\cite{Cambie24}. The main intermediate question asks if there are infinitely many graphs with minimum degree at least $3$, for which the removal of any vertex implies an increase of the total distance. This is essential for possible abundance of \v Solt\'es' graphs, as was the case in the other settings.

\newpage
\subsection*{Overview and outline of the paper}

In~\cref{sec:minimum_USH}, we prove that there is no uniform \v Solt\'es' hypergraph with order bounded by $9.$

We do so by applying different strategies for different pairs $(n,k)$ of order and uniformity on a possible \v Solt\'es' hypergraph $H$.

\begin{itemize}
    \item The case $k=2$ was excluded by a brute force computer check.
    \item If $k \ge n-3 \ge 3$, the connected hypergraph $H \setminus v$ of order $n-1$ satisfies $W(H \setminus v) \le \binom{n-1}{2}+4 < \binom{n}{2} \le W(H)$.
    \item If $k=n-4\ge 3,$ we can exclude hypergraphs with small size by computer verification and for hypergraphs with large size, we check that $W(H \setminus v) < \binom{n}{2} \le W(H)$.
    \item If $k=3$ and $n \in \{8,9\}$, we take a vertex $v$ for which $H \setminus v$ has the largest possible value of $n_1$, the number of adjacent pairs of vertices, and conclude that $W(H \setminus v)<W(H)$ in most cases. For the latter, we notice that two adjacent vertices can be non-adjacent in at most one hypergraph of the form $H \setminus v.$
    The few remaining (sparse) cases in which $W(H \setminus v)<W(H)$ cannot be concluded, are checked by computer.
    \item The case $(n,k)=(9,4)$ is the main case. In this case, $H$ has diameter $2$ and thus the number of adjacent vertices plays an important role in the computation of the total distance $W(H).$
    Here we first prove multiple side results on the transmission and sum of detours (increase of distances due to removal another vertex) in special cases.
    It is observed that adjacent pairs can be non-adjacent in at most two hypergraphs of the form $H \setminus v,$ and equality leads to some structure.
    In the main proof, these observations are combined and it is proven that for every possible value of $W(H)$, $H$ cannot be a \v Solt\'es' hypergraph. 
\end{itemize}

The results on the abundancy of uniform \v Solt\'es' hypergraphs in~\cref{sec:abundancy} and~\cref{sec:irregular} depend on verifying certain constructions. Hereby the particular construction in~\cref{sec:abundancy} is part of a more general construction, called the knits-construction in~\cite{Tiwari25thesis}.

\section{The smallest uniform \v Solt\'es' hypergraphs have order $10$}\label{sec:minimum_USH}

We start excluding most combinations of order and uniformity.

\begin{prop}\label{prop:kunin9}
There is no $k$-uniform \v Solt\'es' hypergraph with order $n$, when
 $n \le 9$ and $k \ge n-4.$ 
\end{prop}

\begin{proof}
By~\cite[Thm. 3]{Cambie24Hyper}, which implies that there is no uniform \v Solt\'es' hypergraph of order at most $5$, we only need to check for $n \ge 6$.
Also $k \ge 3$ as the smallest \v Solt\'es' graph is $C_{11}$, as Dragan Stevanovi\'c checked it is the only \v Solt\'es' graph  of order bounded by $12.$

For $6 \le n \le 9,$ and $k \ge n-3$, by~\cite[Thm.~1]{CGSTT23+} (which determines the structure of the uniform hypergraphs with maximum Wiener index for given order and uniformity) we have that $W(H\setminus v)=+\infty$ or $W(H\setminus v)\le \binom{n-1}{2}+4 < \binom{n}{2}\le W(H) $. 
Next, we consider the cases where $k=n-4.$ All mentioned computer verification can be found at \cite[doc. \text{NoSmallSH}]{C25github}.
 
A computer check can verify that no $3$-uniform hypergraph on $n=7$ vertices with $m \le 10$ edges can be a \v Solt\'es' hypergraph.
When the size is $m \ge 11$, there is a vertex $v$ (with minimum degree) whose deletion leaves at least $\ceilfrac{4m}{7} \ge 7$ edges.
But then $H \setminus v$ is either disconnected, or (by computer check) $W(H \setminus v) \le 19<\binom{7}{2} \le W(H)$.
More precisely, we know that $W(H \setminus v)$ has to be finite and thus that $H \setminus v$ has order $6$ and is $3$-uniform, if two vertices are at distance $3,$ $H \setminus v$ can have at most $6$ edges. Hence $\diam(H \setminus v)$ needs to be $2$ and we can check that there are at most $2\cdot 2=4$ non-adjacent vertices.

 A computer check can verify that no $4$-uniform hypergraph on $n=8$ vertices with $m \le 6$ edges can be a \v Solt\'es' hypergraph.
When the size is $m \ge 7$, there is a vertex $v$ (with minimum degree) whose deletion leaves at least $\ceilfrac{4m}{8} \ge 4$ edges.
 But then $\diam(H \setminus v) \le 2$ and $W(H \setminus v) \le \binom{7}{2}+6<\binom{8}{2} \le W(H)$.

 The remaining case is $n=9$ and $u=5.$
 A computer check can verify that no $5$-uniform hypergraph on $n=9$ vertices with $m \le 6$ edges can be a \v Solt\'es' hypergraph.
 When the size is $m \ge 7$, there is a vertex $v$ (with minimum degree) whose deletion leaves at least $\ceilfrac{4m}{9} \ge 4$ edges.
 But then again $\diam(H \setminus v) \le 2$ and $W(H \setminus v) \le \binom{8}2+6=34<\binom{9}{2} \le W(H)$.
\end{proof}

As such, there are $3$ cases left, $(n,k) \in \{(8,3),(9,3),(9,4)\}.$ These cases are excluded in the following subsections.

\subsection{No $3$-uniform \v Solt\'es' hypergraph with order $8$ or $9$}

\begin{thm}\label{thm:no3unif_n8}
     There is no $3$-uniform \v Solt\'es' hypergraph with order $8$.
\end{thm}

\begin{proof}
    By~\cite[Thm.~1]{CGSTT23+}, there is a unique tight path $P_{7}^3$ which maximises the Wiener index of $3$-uniform hypergraphs of order $7$, and it has Wiener index $38$.
    Hence if $H$ is $3$-uniform \v Solt\'es' hypergraph with order $8$, $W(H)=W(H \setminus v) \le 38.$
    Let $\abs{E(H)}=28-r.$
    Then $W(H) \ge 28+r$, since $r$ pairs of distances are at least two.
    In particular, we have $r \le 10.$
    
    We observe that for every two adjacent vertices $u,v \in v(H)$, there is at most one vertex $w \in V \setminus\{u,v\}$ such that $d_{H \setminus w}(u,v)>1.$
    Every two vertices for which $d_H(u,v)>1$ are both present in $6$ out of the $8$ hypergraphs $H\setminus v_i$ for $1 \le i \le 8.$
    
    By the above two observations, among the $8$ choices of the vertices we can remove, there is at least one vertex $v_i$ for which the number of pairs at distance at least two is bounded by $\frac 34 r + \frac{28-r}{8}.$

    We construct an underlying graph $G$ of $H \setminus v_i$ with $V(G)=V(H \setminus v_i)$ and the edge set consisting of exactly the pairs $\{u,v\}$ where $u,v$ belong to a common hyperedge in $H \setminus v_i.$

    By a result of Solt\'es~\cite[Thm.~1]{Soltes91}, the maximum Wiener index of graphs given their order and size is obtained by the concatenation of a colex graph and a clique, and so one can compute that for a connected graph $G$ of order $7$ and size $m$,
    $$W(G) \le
    \begin{cases}
        42-m& \mbox{ if } m \in [16,21],\\
        58-2m& \mbox{ if } m \in [12,16],\\
        70-3m& \mbox{ if } m \in [9,11].
    \end{cases}$$

    When $m\ge 21-\left(\frac 34 r + \frac{28-r}{8}\right)=17.5-\frac{5r}{8},$
    we notice that $W(H \setminus v_i)=W(G) < 28+r \le W(H)$.
    For the latter, it is actually sufficient to notice this for $r=10$, in which case $m \ge 12$, and thus $34<38$.
    Hence $H$ cannot be a Solt\'es' hypergraph.    
\end{proof}

With a bit more case analysis, the case $(n,k)=(9,3)$ is excluded as well along the same lines (and a little bit of computer verifications).

\begin{thm}\label{thm:no3unif_n9}
     There is no $3$-uniform \v Solt\'es' hypergraph with order $9$.
\end{thm}

\begin{proof}
    By~\cite[Thm.~1]{CGSTT23+}, there is a unique tight path $P_{8}^3$ which maximises the Wiener index of $3$-uniform hypergraphs of order $7$, and it has Wiener index $57$.
    Hence if $H$ is $3$-uniform \v Solt\'es' hypergraph with order $8$, $W(H)=W(H \setminus v) \le 57.$
    Let $\abs{E(H)}=36-r.$
    Then $W(H) \ge 36+r$, since $r$ pairs of distances are at least two.
    In particular, we have $r \le 21.$

    Among the $9$ choices of a vertex to be removed, there is at least one for which $H \setminus v_i$ has size at least $28-\left( \frac79 r+\frac{36-r}9\right)=24-\frac 23 r.$
    Let $G$ the underlying graph of $H \setminus v_i.$
    If $r \le 16,$ we check that $W( H \setminus v_i) \le W(G) < 36+r.$

    The remaining cases where $W( H \setminus v_i) \in [53,57]$ are easily excluded with a final computer check as well.
    For this, we check the plausible size of $H\setminus v_i$ and check the hypergraphs for which the removal of any vertex can have such a size. See~\cite[doc. \text{NoSoltes(9,3)}]{C25github}.
\end{proof}

\subsection{No $4$-uniform \v Solt\'es' hypergraph with order $9$}

We start by proving a few easy lemmas.

\begin{lem}\label{lem:sizediam3}
    Let $H$ be a connected $4$-uniform hypergraph of order $8$.
    Then its size is at least $3$ and its diameter is at most $3.$
\end{lem}

\begin{proof}
 Two hyperedges only cover $8$ vertices if they are disjoint, in which case the hypergraph is not connected. Hence the size of $H$ has to be at least $3.$

    If $u$ and $v$ are two vertices for which $d(u,v) \ge 3$, then the hyperedges containing $u$ resp. $v$ are disjoint. Since $8=2 \cdot 4,$ this implies that both are covered by a single hyperedge, and these two hyperedges are disjoint.
    Any other hyperedge needs to intersect both its edges, which implies $d(u,v) \le 3.$
    Since $u$ and $v$ were chosen arbitrarily, we conclude.
\end{proof}

We will use the following definition in the remaining of this subsection.

\begin{defi}
    For a hypergraph $H$, let $n_i$ be the number of pairs of vertices in $H$ for which the distance between them equals $i$.
\end{defi}

\begin{lem}\label{lem:sizeGE15}
    A connected $4$-uniform hypergraph $H$ of order $8$ satisfies $n_1 \ge 15$. Equality can only be attained by hypergraphs of size $3.$
\end{lem}
\begin{proof}
    Let $e_1$ be a hyperedge of $H$. Then $e_1$ already contains $6$ pairs of vertices at distance one.
    Since $H$ is connected, we can list the other hyperedges as $e_2, e_3, 
    \ldots, e_m$ such that $e_1 \cap e_2 \not= \emptyset.$
    Let $H_i$ be the hypergraph spanned by the edges $e_1, e_2, \ldots, e_i.$
    If $\abs{V(H_i)}-\abs{V(H_{i-1})}=t_i,$ then there are at least $f(t_i)$, where $f(t)=\binom{t}{2}+t(4-t)$, additional pairs of vertices at distance $1$ more in $H_i$ than in $H_{i-1}$.
    For this, note that every pair of vertices in $H_i \setminus H_{i-1}$ are at distance $1$, and also every $u \in V(H_i)\setminus V(H_{i-1})$ and $v \in V(H_{i-1}) \cap e_i$ are at distance $1$ (and these pairs were not counted when considering $H_{i-1}.$
    Since $f(1)=3,f(2)=5$ and $f(3)=6$, we easily deduce that $\sum f(t_j) \ge 9$ (and thus $n_1 \ge 6+9=15$) for nonnegative integers $t_j$ satisfying $\sum_{j=2}^m t_j=4$, and equality is only obtained by $f(1)+f(3).$ 
    Hence there are three hyperedges covering the $8$ vertices when equality occurs. Without loss of generality, these are $e_1,e_2, e_3.$
    If the size is at least four, there is also an edge $e_4.$
    Let $k$ be minimum such that $H_k$ contains the vertices of $e_4$ and $v$ be a vertex from $e_4$ in $ V(H_k)\setminus V(H_{k-1})$.
    Then as $e_4 \not= e_k,$ $e_4$ should connect $v$ with a non-adjacent vertex in $V(H_{k-1})$, and thus our lower bound $6+f(1)+f(3)$ is not sharp.    
\end{proof}

\begin{cor}\label{cor:maxWfordiam2}
    A connected $4$-uniform hypergraph $H$ of order $8$ and diameter $2$ satisfies $W(H) \le 41.$
    Equality can only be obtained by hypergraphs of size $3.$
\end{cor}

\begin{proof}
    If $H$ has diameter $2$, then $W(H)=n_1+2n_2=n_1+2\left( \binom82 - n_1 \right) =56-n_1.$ The latter is bounded by $41$ by~\cref{lem:sizeGE15}, and the claim also says equality can only occur if the size is $3.$
\end{proof}

\begin{lem}\label{lem:Wle44}
    A connected $4$-uniform hypergraph $H$ of order $8$ and diameter $3$ satisfies $W(H) \le 44.$
    Equality is obtained if and only if $H$ has size $3.$
\end{lem}

\begin{proof}
    Since $\diam(H)=3,$ it has two disjoint hyperedges $e_1$ and $e_2$.
    A third hyperedge intersects them in $1$ and $3$ vertices, or both in $2$ vertices.
    In both cases, $W(H)=44$, since $4 \cdot 3+ 16\cdot 1+ 8 \cdot 2=3 \cdot 3+ 15\cdot 1+ 10 \cdot 2$.
    If there are more than $3$ hyperedges, the total distance is clearly smaller after adding further edges (since $n_1$ increases and no distance can increase).
\end{proof}


 \begin{lem}\label{lem:Wbound}
        A connected $4$-uniform hypergraph $H$ of order $8$ and size at least $4$, satisfies either 
        \begin{itemize}
            \item $\diam (H) \le 2$ and $W(H) \le 40$
            \item $\diam (H)  =3$, $W(H)\le 41$ and $H$ has exactly one pair of vertices at distance $3$ 
            \item $\diam(H)  =3$, $40\le W(H) \le 42$ and $H$ has exactly two pairs of vertices at distance $3$ 
        \end{itemize}
    \end{lem}
    \begin{proof}
        Note that vertices at distance $3$ need to belong to disjoint hyperedges (and have degree one).
        Let $e_1$ and $e_2$ be those disjoint hyperedges. Let $e_i$ have $x_i$ vertices of degree $1.$
        Since the size is at least $4$, two other hyperedges cover at least $5$ vertices, thus $x_1+x_2 \le 3$ and hence $x_1 x_2 \le 2.$
        This implies we have exactly those three cases.

        Now $W(H)=3n_3+2n_2+n_1= 3n_3+2(28-n_3-n_1)+n_1=56+n_3-n_1.$
        \Cref{lem:sizeGE15} implies that $n_1 \ge 16,$ and the bounds thus follow for $n_3 \in \{0,1,2\}.$

        Finally, the lower bound when $n_3=2$ and thus $x_1+x_2=3$ is obtained from the construction where all other $5$ vertices are at distance one from each other. In this case $n_2+n_3 \ge 3 \cdot 4-2=10$ and thus $n_1 \le 18,$ resulting into $W(H)=56+n_3-n_1\ge 56+2-18=40.$
    \end{proof}

\begin{thm}\label{thm:no4unif_n9}
     There is no $4$-uniform \v Solt\'es' hypergraph with order $9$.
\end{thm}

\begin{proof}
    Assume the contrary and let $H$ be a $4$-uniform \v Solt\'es' hypergraph with order $9$.

    Observe that $\diam(H)\le 2,$ since $\delta(H)\ge 2$ and no two vertices can have disjoint closed neighbourhoods, since $\abs{N[v]}\ge 5$ for every $v\in V(H).$

    \begin{claim}\label{clm:claimdeg2}
        If $\deg(v)=2,$ then $\sigma(v)\ge 10$. 
    \end{claim}
    \begin{claimproof}
        Since $v$ is contained in $2$ edges, it has $r \le 6$ neighbours (equality in the case where $v$ is the only common vertex of these two edges). Thus, $v$ will be at a distance of $1$ from $r$ vertices and $2$ from the remaining $8-r$. This means $\sigma(v)\ge r\times 1+(8-r)\times2=16-r \ge 10$.
    \end{claimproof}

    Let $\deg_H(uv)$ be the number of hyperedges in $E(H)$ containing both $u$ and $v$. If $v$ is fixed, we will call $\deg_H(uv)$ the multiplicity of $u$.

    The following claim immediately follows from the fact that at most one $4$-uniform hyperedge can contain four fixed vertices.

\begin{claim}\label{clm:deg(uv)gexuv-1}
    If $\deg_H(uv)\ge 2$, then there is at most one vertex $w \in V \setminus \{u,v\}$ which belongs to all hyperedges containing $u$ and $v$.
\end{claim}

    \begin{claim}\label{clm:claimdeg3}
        If $\deg(v)=3$, then either $\sigma(v)=8$ and the unique $u \in V \setminus v$ with $\deg_H(uv)= 2$ satisfies $d_H(u,v)=d_{H \setminus w}(u,v)$ for every $w \in V \setminus \{u,v\},$ or $\sigma(v)\ge 9$ and there is at least one $u \in V \setminus v$ such that $\deg_H(uv)\ge 2$.
    \end{claim}
    \begin{claimproof}
        If $v$ has eight neighbours (and thus $\sigma(v)=8$), there is exactly one vertex $u$ belonging to two of the hyperedges $e_1, e_2$ containing $v$, and $e_1 \cap e_2 =\{u,v\}$ as all vertices in $V \setminus \{u,v\}$ have exactly one common hyperedge with $v$. This implies $d_H(u,v)=d_{H \setminus w}(u,v)=1$ for every $w \in V \setminus \{u,v\}.$

        If $v$ has less than eight neighbours, then $\sigma(v)\ge 9$ and there is at least one $u \in V \setminus v$ such that $\deg_H(uv)\ge 2$ (by the pigeon hole principle, since the sum of multiplicities of $8$ vertices is $9$).     
    \end{claimproof}

    \begin{claim}\label{clm:claimdeg4}
        If $\deg(v)\ge 4,$ there are at least two neighbours $u_1,u_2 \in V \setminus v$ such that $\deg_H(u_iv)\ge 2$ for $i=1,2$.
    \end{claim}
    \begin{claimproof}
        Assume the contrary, then at most one vertex $u$ of $H \setminus v$ satisfies $\deg_H(uv)\ge 2$. Hence every hyperedge containing $v$ contains at least two vertices $w$ for which $\deg_H(vw)=1.$
        This implies there are at least $8$ vertices sharing only one hyperedge with $v$. Since $\abs{H \setminus v}=8$ and equality is impossible, we conclude.
    \end{claimproof}

Since $H$ is a \v Solt\'es' hypergraph, $H\setminus v$ is a connected $4$-uniform hypergraph of order $8$ and thus satisfies $36=\binom 92 \le W(H)=W(H\setminus v)\le44$ by~\cref{cor:maxWfordiam2} and~\cref{lem:Wle44}.

For every pair of adjacent or non-adjacent vertices $u,v,$ we count the number $x_{uv}$ of vertices $w \in V \setminus\{u,v\}$ for which $d_H(u,v)<d_{H \setminus w}(u,v).$

Observe that by definition
\begin{equation} \label{eq:wiener-index}
    W(H \setminus v_i) = W(H) - \sigma(v_i) + \sum_{\substack{u, w \in V \\ u, w \neq v_i}} \big[d_{H \setminus v_i}(u, w) - d_H(u, w)\big]
\end{equation}
and since $H$ is a \v Solt\'es' hypergraph,
\begin{equation} \label{eq:wiener-index2}
        2W(H)=\sum_{i=1}^{9}\sigma(v_i) = \sum_{i=1}^{9}\sum_{\substack{u, w \in V \\ u, w \neq v_i}} \big[d_{H \setminus v_i}(u, w) - d_H(u, w)\big]
\end{equation}

We now end the proof with a small case-analysis.
If $W(H)=44$, then by~\cref{lem:Wle44}, we know $W(H\setminus v)=44$ implies that $\abs{E(H\setminus v)}=3$ for every $v \in V(H)$ and thus $H$ is $d$-regular, where $d$ satisfies $\frac{9d}{4}-d=3$ and thus $d=\frac{12}{5},$ which is impossible.

If $37\le W(H)\le 43$, we write $W(H)=\binom{9}{2}+r$, where $r=n_2$ satisfies $1\le r\le 7$.

From claims~\ref{clm:claimdeg2},~\ref{clm:deg(uv)gexuv-1},~\ref{clm:claimdeg3} and~\ref{clm:claimdeg4}, we know that $(\sigma(v)-8)+\sum_{u \in N(v)} (2-x_{uv})\ge 2$ for every vertex $v$.
Summing over all $9$ vertices in $V(H)$, results into
$2(36+r)-9\cdot 8+2\sum_{uv \in E(H)}(2-x_{uv})\ge 2\cdot 9$ and thus $\sum_{uv \in E(H)} x_{uv}\le 63-r.$

If $y$ is an upper bound for the number of pairs of vertices at distance $3$ of each other in $H \setminus v_i$, then
$$\sum_{i=1}^{9}\sum_{\substack{u, w \in V \\ u, w \neq v_i}} \big[d_{H \setminus v_i}(u, w) - d_H(u, w)\big] \le 9y+ \sum_{uv \in E(H)} x_{uv}. $$

By~\cref{cor:maxWfordiam2},~\cref{lem:Wle44} and~\cref{lem:Wbound}, we conclude  that~\eqref{eq:wiener-index2} is not satisfied since $2(36+r) > 63-r+18=81-r$ for $4 \le r$ and $2(36+r) > 63-r+9=72-r$ when $1 \le r \le 3$.

For $r=0,$ the last inequality becomes an equality. Every vertex $v$ needs to satisfy $\diam(H \setminus v)=3$ and thus $d_{H \setminus v}(u,w)=3$ for two vertices $u,w \in V.$
Since $W(H)=36$, $\deg(v) \ge 4$ and the equality implies that $u,w$ are the only two vertices sharing multiple hyperedges with $v$ and further that at least two hyperedges contain $\{u,v,w\}$ and thus also $\deg_H(vw) \ge 2.$
Since the properties for $u,w,v$ are the same, there should be three disjoint hyperedges for which $\abs{ e \cap \{u,v,w\}}=1$, which is impossible. We conclude that no equality is possible in this case.
\end{proof}

\section{Abundancy of uniform \v Solt\'es' hypergraphs}\label{sec:abundancy}

In this section, we prove that there exist uniform \v Solt\'es' hypergraphs of any order $n \ge 10$.

For $10 \le n \le 100$, examples were discovered with multiple constructions. See e.g.~\cite[Sec. 3 \& app A]{Tiwari25thesis} for examples.

Given $n\geq 100$, we can find $s,t\in \mathbb{N}$ such that $n=\binom{s}{2}-t,\quad 0\leq t\leq (s-2)$.
Let $k=n-(t+2s+1)$, and let $H$ be the hypergraph with vertex set $V=\{0,1,2,...,n-1\}$ and edge set $E=\{e_i \mid 0 \le i \le n-1\}$,
where $e_i= \{ i, i+2s+k-1\} \cup \{i+s+1, i+s+2,\ldots, i+s+k-2 \}$ (vertices are considered modulo $n$).

By definition, $H$ is a vertex-transitive, $k$-uniform hypergraph of order $n$.

Note that $k-2 \ge \frac n2$ since $\binom s2 \ge 7s \ge 2(2s+t+1)+t+4$ ($s \ge 15$ since $n>91$).
As such, it is trivial that $\diam(H\setminus v) \le 2$ $\forall v\in V(H)$ and easily seen that $\diam(H)=1.$

We now prove the following lemma.

\begin{lem}
    There are exactly $n-1$ pairs of non-adjacent vertices in $H\setminus v.$
\end{lem}

\begin{proof}
    Since $H$ is vertex-transitive, one can check this for $v=0.$
    
    The edge $e_{\ceilfrac{t}{2}}$ covers every vertex (and thus combination of vertices in this range) $\ceilfrac{t}{2}+s+1 \le j \le \ceilfrac{t}{2}+s+k-2=n-(\floorfrac{t}{2}+s+1).$
    The edge $e_{-s}$ and $e_{t+s+2}$ cover all vertices $1 \le j \le \ceilfrac{t}{2}+s+1 $, resp. $ n-(\ceilfrac{t}{2}+s+1) \le j \le -1 $.
    For $s+1 \le i \le \floorfrac{t}{2}+s+1,$ the edge $e_i$ covers every $ n-( \floorfrac{t}{2}+s+1) \le j\le n-s$.

    Finally, let $(n-i,j) \in [n-s,\ldots,n-1] \times [1,s]$ (here $1 \le i,j \le s$).
    If $i+j\ge s+1$, the edge $e_{n-i}$ contains both $n-i$ and $j$.

    If $i+j \le s$, and thus $n-i-j>k-2+s,$ the only edge which might contains $n-i$ and $j$ but not $0$, is $e_j.$
    This happens precisely if $i+j=t+2.$
    Hence the number of non-neighbours in $H \setminus v$ of $j$ is $s-j-1$ if $1 \le j \le t+1$ and $s-j$ if $t+2 \le j \le s.$
    We conclude that the number of pairs of vertices in $H \setminus v$ which are distance $2$ apart equals
     $$\sum\limits_{j=1}^{t+1}(s-j-1) + \sum\limits_{j=t+2}^{s-1} (s-j)
     = \binom{s}{2}-t-1=n-1.$$
     \end{proof}
    Hence 
    $$W(H \setminus v) = \binom{n-1}2+n-1=\binom n2= W(H) $$
    and thus $H$ is a \v Solt\'es' hypergraph.

    So we have found uniform \v Solt\'es' hypergraph of every order $n \ge 10.$

    \begin{remark}
        The above construction provides uniform \v Solt\'es' hypergraphs for $75\%$ of the possible uniformities as well.
        The construction can be generalised to one for every $r \ge 1$, by having hyperedges $e_i=[i-s-r..i-s-1] \cup [i..i+k-2r]\cup[i+k-2r+s+1..i+k-r+s]$ (again indices modulo $n$) for $n=\binom{s}{2}-(2r-1)t-r^2+1$ and $k=\binom{s}{2}-2rt-2s-r^2$ for $0 \le t < s-\binom{r+1}{2}$. For $4 \le k \le 10^7$, there were only $8$ uniformities missing; $\{905, 1301, 1721, 2801, 9641, 11969, 52109, 5335709\}.$ 
        Since there are many other adaptations of the construction possible, for every uniformity at least $4$, one can expect that there exist $k$-uniform \v Solt\'es' hypergraphs. 
        We decided not to work out these other cases, as it does not provide new insights in the main question by \v Soltes or other things.

        No $3$-uniform example has been found, so one subquestion of~\cite[Ques.~13]{Cambie24Hyper} remains open.
    \end{remark}

\section{An irregular $9$-uniform \v Solt\'es' hypergraph}\label{sec:irregular}

In this short subsection, we present an example addressing the fifth subquestion of~\cite[Ques.~13]{Cambie24Hyper}.

\begin{exam}
Let $H$ be the $9$-uniform hypergraph with vertex set $V=[54]=\{1,2,\ldots,54\}$
and $27$ hyperedges given by
$\{ \{a,a+1,a+2,a+3,a+4,a+5,a+7,a+16,a+18\} \mid a \in [54], a \equiv 0 \pmod 2\}.$
Here elements have to be interpreted modulo $54.$
Then $H$ is an irregular uniform \v Solt\'es' hypergraph.
\end{exam}

\begin{proof}

The irregularity is easy to check, since $54 \nmid 27 \times 9$ and more precisely one can observe that there are $5$ even and $4$ odd elements in every edge.
Due to the symmetry of the construction, the hypergraph $H$ has two vertex orbits, corresponding to the parity classes, and these vertices have degree $4$ (for odd vertices) and $5$ (the even vertices).

Now one can check that $W(H)=2349=W(H \setminus 1)=W(H \setminus 2)$ ($1$ and $2$ being representatives of the odd and even vertices).

This is also verified in~\cite[\texttt{UnifSoltesHypergraphs/Verification\_IrregularUSH}]{C25github}.
\end{proof}


\end{document}